\documentclass[a4paper,12pt]{article}
\usepackage{cmap}
\usepackage[cp1251]{inputenc}
\usepackage[english]{babel}
\usepackage[left=2cm,right=2cm,top=2cm,bottom=2cm]{geometry}
\usepackage{amssymb}
\usepackage{amsmath, amsthm}
\theoremstyle{plain}
\newtheorem{theorem}{Theorem}

\newtheorem{corollary}{Corollary}[theorem]
\newtheorem{conj}{Conjecture}

\theoremstyle{definition}

\theoremstyle{remark}
\newtheorem{remark}{Remark}

\begin{document}

\begin{center}\Large
\textbf{On a generalization of one  Kramer’s theorem}
\normalsize

\smallskip
Viachaslau I. Murashka 

 \{mvimath@yandex.ru\}

Faculty of Mathematics and Technologies of Programming,

 Francisk Skorina Gomel State University,  Gomel 246019, Belarus\end{center}

\begin{abstract}
  Yangming~Li and Xianhua~Li in 2012 proposed a conjecture that generalizes O.U.~Kra\-mer's result about supersoluble groups. Here we proved that this conjecture is false in the general case and true for groups with the trivial Frattini subgroup.
\end{abstract}


\textbf{AMS}(2010). 20F16  (Primary)  20F17,   20D25 (Secondary).

\section*{Main results}

  All groups considered here will be finite. According to B. Huppert's Theorem a group is supersoluble if and only if all its maximal subgroups have prime indexes. Recall that the Fitting subgroup $\mathrm{F}(G)$ of a group $G$ is the largest normal nilpotent subgroup of $G$.    From O.U.~Kramer's result \cite{f5}  it follows that a soluble group is supersoluble if and only if all its maximal subgroups that do not contain the Fitting subgroup have prime indexes. In the universe of all groups the  Fitting subgroup does not have many properties which it has in the soluble universe. Recall \cite[X, Theorem 13.13]{19}  that \emph{the generalized Fitting subgroup} $\mathrm{F}^*(G)$ can be defined by
 $ \mathrm{F}^*(G)/\mathrm{F}(G)=\mathrm{Soc}(\mathrm{F}(G)C_G(\mathrm{F}(G))/\mathrm{F}(G)).$
P. Schmid \cite{f3} and L.A.~Shemetkov
 \cite[Definition~7.5]{f4} considered another generalization $\tilde{\mathrm{F}}(G)$ of the Fitting subgroup defined by     $ \Phi(G)\subseteq \tilde{\mathrm{F}}(G)$ and
  $\tilde{\mathrm{F}}(G)/\Phi(G)=\mathrm{Soc}(G/\Phi(G))$. They proved that $C_G(\tilde{\mathrm{F}}(G))\subseteq \tilde{\mathrm{F}}(G)$ in every group $G$. 
  Yangming~Li and Xianhua~Li \cite{30} extended O.U.~Kramer's result to the universe of all groups. They proved: a   group $G$ is supersoluble if and only if all its maximal subgroups that do not contain $\tilde{\mathrm{F}}(G)$ have prime indexes. Also they showed that  the previous result is false if we replace $\tilde{\mathrm{F}}(G)$ by $\mathrm{F}^*(G)$. Another generalization of O.U.~Kramer's result was proposed earlier by Yanming  Wang et.\,al. \cite{Wang}. Their result states that a   group $G$ is supersoluble if and only if all its maximal subgroups that do not contain $\mathrm{F}(H)$ have prime indexes where $H$ is a normal soluble subgroup of $G$ and $G/H$ is supersoluble. Recall that the formation of all supersoluble groups is denoted by $\mathfrak{U}$.
  Yangming Li and Xianhua Li proposed the following conjecture.
  
  \begin{conj}[{\cite[Conjecture 1]{30}}]\label{conj1}
  Let $\mathfrak{F}$ be a saturated formation containing $\mathfrak{U}$ and suppose
that $H$ is a normal subgroup of $G$ such that $G/H\in \mathfrak{F}$. If, for any maximal
subgroup $M$ of $G$, there holds that $|\mathrm{\tilde F}(H) : \mathrm{\tilde F}(H) \cap M| = 1$ or a prime, then $G\in \mathfrak{F}$.
  \end{conj}
  
  In this paper we give the negative answer on this conjecture.
  
  \begin{theorem}
    Conjecture \ref{conj1} is false for any saturated formation $\mathfrak{F}$ of soluble groups containing $\mathfrak{U}$. In particular, it is false for $\mathfrak{U}$.
  \end{theorem}
  
  \begin{proof}
    Let $K\simeq A_5$ be the alternating group of degree 5 and  $V$ be the permutation $\mathbb{F}_5K$-module. Then the dimension of $\mathrm{Soc}(V)$ is 1. Let $W\simeq V/\mathrm{Soc}(V)$. Note that $W$ is an indecomposable module,  $\mathrm{dim}(W)=4$, $\mathrm{Rad}(W)$ is a faithful simple module, $\mathrm{dim}(\mathrm{Rad}(W))=3$, and $W/\mathrm{Rad}(W)$ is a trivial module.
    Now let $G=W\rtimes K$. Hence $\mathrm{Rad}(W)\leq \Phi(G)$ by \cite[B, Lemma 3.14]{s8}. Note that $H=\mathrm{Rad}(W)K\trianglelefteq G$ and $G/H\simeq Z_5\in\mathfrak{U}\subseteq\mathfrak{F}$.   Since $\mathrm{Rad}(W)$ is a faithful simple module, we see that $\mathrm{Rad}(W)$ is the unique minimal normal subgroup of $H$ and $\Phi(H)=1$. Hence $\mathrm{\tilde F}(H)=\mathrm{Rad}(W)\leq \Phi(G)$. Now $|\mathrm{\tilde F}(H):\mathrm{\tilde F}(H)\cap M|=1$ for every maximal subgroup $M$ of $G$.
    So $G$ satisfies the statement of Conjecture~\ref{conj1}. Since $G$ is not soluble, $G\not\in\mathfrak{F}$. Hence Conjecture~\ref{conj1} is false for any saturated
    formation $\mathfrak{F}$ of soluble groups containing $\mathfrak{U}$.
  \end{proof}

\begin{remark}

The module from the previous theorem can be constructed in GAP \cite{GAP} by the following commands:

  $V:=PermutationGModule(AlternatingGroup(5), GaloisField(5));$

  $L:=MTX.BasisSocle(V);$

  $W:=MTX.InducedActionFactorModule(V, L);$

 One can check that this module is indecomposable and find all its composition factors with the following commands:

  $MTX.IsIndecomposable(W);$

  $MTX.CompositionFactors(W);$
\end{remark}

Our counterexample is based on $\Phi(G)\neq 1$. Nevertheless if $\Phi(G)\cap H\leq \Phi(H)$, then the statement of the conjecture is true.

\begin{theorem}\label{thm1}
  Let $\mathfrak{F}$ be a saturated formation containing $\mathfrak{U}$ and suppose
that $H$ is a normal subgroup of $G$ such that $G/H\in \mathfrak{F}$ and $\Phi(G)\cap H\leq \Phi(H)$. If, for any maximal
subgroup $M$ of $G$, there holds that $|\mathrm{\tilde F}(H) : \mathrm{\tilde F}(H) \cap M| = 1$ or a prime, then $G\in \mathfrak{F}$.
  \end{theorem}

\begin{proof}
Assume that the statement of Theorem~\ref{thm1} is false and let a group $G$ be a minimal order counterexample.

$(a)$ $\Phi(H)=1$. 

Suppose that $\Phi(H)\neq 1$. 
Since $H\trianglelefteq G$, wee see that $\Phi(H)\leq\Phi(G)$. Note that $G/H\simeq (G/\Phi(H))/(H/\Phi(H))\in\mathfrak{F}$, $\mathrm{\tilde F}(H/\Phi(H))=\mathrm{\tilde F}(H)/\Phi(H)$ and $M$ is a maximal subgroup of $G$ iff $M/\Phi(H)$ is a maximal subgroup of $G/\Phi(H)$. Now the hypothesis of Theorem~\ref{thm1} holds for $G/\Phi(H)$. By our  assumption $G/\Phi(H)\in\mathfrak{F}$. Since $\mathfrak{F}$ is a saturated formation and $\Phi(H)\leq\Phi(G)$, we see that $G\in\mathfrak{F}$, a contradiction.

$(b)$ \emph{$\mathrm{\tilde F}(H)$ is abelian}. 

Assume that $\mathrm{\tilde F}(H)=\mathrm{Soc}(H)$ is non-abelian.  Then there is a minimal normal non-abelian subgroup $N$ of $G$ with $N\leq \mathrm{\tilde F}(H)$.  Note that there is a simple group $S$ such that    $N\simeq N_1\times\dots\times N_k$ and $N_i\simeq S$.  Let $p$ be the largest prime divisor of $|S|$ and $P$ be a Sylow $p$-subgroup of $N$. It is clear that $N_G(P)\neq G$. Let $M$ be a maximal subgroup of $G$ with $N_G(P)\leq M$.     Then $G=N_G(P)N=MN=M\mathrm{\tilde F}(H)$ by Frattini argument.
Now $|G|=\,|M||N:M\cap N|=|M||\mathrm{\tilde F}(H) : \mathrm{\tilde F}(H) \cap M|  $. By our assumption $|N:M\cap N|$ is a  prime $q$. Since $P\leq N_N(P)\leq M\cap N$, we see that $p\neq q$.
According to \cite[Theorem 1.1]{Isaac} $N/\mathrm{Core}_N(N\cap M)$ is  isomorphic to a subgroup of  the symmetric group of degree $q$. Note that $N/\mathrm{Core}_N(N\cap M)$  is isomorphic to a direct product of positive number of isomorphic to $S$ groups. Hence every prime divisor of $|S|$ is not greater than $q$, a contradiction with $p>q$.

$(c)$ \emph{$\mathrm{\tilde F}(H)$ is a direct product of minimal normal subgroups of $G$ and their orders are primes}.

According to The Krull-Remark-Shmidt Theorem $\mathrm{\tilde F}(H)$ admits a direct decomposition into directly $G$-indecomposable normal subgroups of $G$. Let $T$ be one of them and $N$ be a minimal normal subgroup of $G$  below $T$. From $\Phi(G)\cap H\leq\Phi(H)=1$, we see that  there is a maximal subgroup $M$ of $G$ with $G=MN=M\mathrm{\tilde F}(H)$. Since $N$ is abelian, we see that $M\cap N\trianglelefteq G$. Hence $M\cap N=1$. Now $|G|=|M||N|=|M||\mathrm{\tilde F}(H) : \mathrm{\tilde F}(H) \cap M|$. It means that $|N|=$ $|\mathrm{\tilde F}(H) : \mathrm{\tilde F}(H) \cap M|$ is a prime by our assumption.
Since $\mathrm{\tilde F}(H) \cap M$ is abelian and  $G=M\mathrm{\tilde F}(H)$, we see that  
$\mathrm{\tilde F}(H) \cap M \trianglelefteq G$. Now $T=\mathrm{\tilde F}(H) 
\cap T=N\times(\mathrm{\tilde F}(H) \cap M )\cap T=N\times(\mathrm{\tilde F}(H) \cap M \cap T)$  
and $(\mathrm{\tilde F}(H) \cap M \cap T)\trianglelefteq G$. Since $T$ is  directly $G$-indecomposable, we see that $T=N$ is a minimal normal subgroup of $G$ of a prime order.  

\newpage

$(d)$ $G/\mathrm{\tilde F}(H)\in\mathfrak{F}$.

According to $(c)$  $\mathrm{\tilde F}(H)=N_1\times\dots \times N_m$ where $N_i$ is a minimal normal subgroup of $G$ of prime order. Hence $G/C_G(N)$ is isomorphic to a subgroup of the automorphism group of $N$. In particular,  $G/C_G(N)$ is abelian. From $C_G(\mathrm{\tilde F}(H))=\cap_{i=1}^mC_G(N_i)$ it follows that $G/C_G(\mathrm{\tilde F}(H))$ is abelian. 
Since $\mathrm{\tilde F}(H)$ is abelian, we see that $\mathrm{\tilde F}(H)\leq C_H(\mathrm{\tilde F}(H))\leq \mathrm{\tilde F}(H)$. So $\mathrm{\tilde F}(H)=C_H(\mathrm{\tilde F}(H))$.
Since $\mathfrak{F}$ is a formation containing all supersoluble groups and $G/H\in\mathfrak{F}$, we see that $G/(H\cap C_G(\mathrm{\tilde F}(H)))=G/C_H(\mathrm{\tilde F}(H))=G/\mathrm{\tilde F}(H)\in\mathfrak{F}$.

$(e)$ \emph{The final contradiction.}

 From $(c)$ it follows that $\mathrm{\tilde F}(H)$ lies in the supersoluble hypercenter $\mathrm{Z}_\mathfrak{U}(G)$ of $G$. Since $\mathfrak{U}\subseteq \mathfrak{F}$, we see that  $\mathrm{\tilde F}(H)$ lies in the $\mathfrak{F}$-hypercenter of $G$. From $(d)$ it follows that $G=\mathrm{Z}_\mathfrak{F}(G)$.
 Since $\mathfrak{F}$ is saturated, $G\in\mathfrak{F}$, the final contradiction.
\end{proof}

\begin{corollary}[{\cite[Theorem 1.1]{30}}]\label{Cor1}
  Let $G$ be a   group. Then $G$ is supersolvable
if and only if, for any maximal subgroup $M$ of $G$, there holds that $|\mathrm{\tilde F}(G):\mathrm{\tilde F}(G)\cap M| = 1$ or a prime.
\end{corollary}

\begin{proof} Directly follows from Theorem \ref{thm1} for $H=G$.
\end{proof}

\begin{corollary}[Kramer \cite{f5}]\label{kr}
   A soluble group $G$ is supersoluble if and only if   $\mathrm{F}(G)\leq M$ or $M\cap\mathrm{F}(G)$ is a maximal subgroup of $\mathrm{F}(G)$ for every maximal subgroup $M$ of $G$. \end{corollary}

\begin{proof}
  Since $\mathrm{F}(G)$ is nilpotent, we see that if   $M\cap\mathrm{F}(G)$ is a maximal subgroup of $\mathrm{F}(G)$, then $|\mathrm{F}(G):M\cap\mathrm{F}(G)|\in\mathbb{P}$. Note that $\mathrm{F}(G)=\mathrm{\tilde F}(G)$. Now  $G\in\mathfrak{U}$ by Corollary \ref{Cor1}.
\end{proof}

\begin{corollary}[{\cite[Theorem 3.1]{Wang}}]
 Let $\mathfrak{F}$  be a saturated  formation containing $\mathfrak{U}$, $G$ be a group  with a  solvable normal subgroup  $H$  such  that  $G/H \in\mathfrak{F}$.  If for  any maximal  subgroup $M$ of  $G$, either  $\mathrm{F}(H) \leq   M$  or $\mathrm{F}(H) \cap M$ is a maximal subgroup of $\mathrm{F}(H)$,  then  $G \in  \mathfrak{F}$. The  converse also holds, in the  case where $\mathfrak{F} = \mathfrak{U}$.
\end{corollary}

\begin{proof}
  Note that $(G/\Phi(G))/(H\Phi(G)/\Phi(G))\simeq G/H\Phi(G)\in\mathfrak{F}$. Let $K/\Phi(G)=\mathrm{F}(H\Phi(G)/\Phi(G))$. It means that $K=\mathrm{F}(H\Phi(G))$ by \cite[A, Theorem 9.3(c)]{s8}. So
  $K=\mathrm{F}(H\Phi(G))\cap H\Phi(G)=\Phi(G)(K\cap H)=\Phi(G)\mathrm{F}(H)$. Let $M/\Phi(G)$ be a maximal subgroup of $G/\Phi(G)$. Then $M$ is a maximal subgroup of $G$.   Note that $K/\Phi(G)\simeq \mathrm{F}(H)/(\mathrm{F}(H)\cap \Phi(G))$ and $K/\Phi(G)\cap M/\Phi(G)\simeq  (\mathrm{F}(H)\cap M)/(\mathrm{F}(H)\cap\Phi(G))$. It means that $|\mathrm{F}(H\Phi(G)/\Phi(G)):\mathrm{F}(H\Phi(G)/\Phi(G))\cap M/\Phi(G)|\in\mathbb{P}\cup\{1\}$.  So $G/\Phi(G)\in\mathfrak{F}$ by Theorem \ref{thm1}. Since $\mathfrak{F}$ is saturated, $G\in\mathfrak{F}$.
\end{proof}

\subsection*{Acknowledgments}

I am grateful to A.F. Vasil'ev for helpful discussions on this paper.

\end{document}